\documentclass{amsart}
\usepackage{graphicx}
\newtheorem{theorem}{Theorem}[section]
\newtheorem{lemma}[theorem]{Lemma}

\newtheorem{proposition}[theorem]{Proposition}
\newtheorem{problem}[theorem]{Problem}

\theoremstyle{definition}

\newtheorem{example}[theorem]{Example}

\newtheorem{claim}[theorem]{Claim}
\newtheorem{conjecture}[theorem]{Conjecture}

\theoremstyle{remark}
\newtheorem{remark}[theorem]{Remark}

\numberwithin{equation}{section}

\begin{document}

\title{Waist and trunk of knots}

\author{Makoto Ozawa}
\address{Department of Natural Sciences, Faculty of Arts and Sciences, Komazawa University, 1-23-1 Komazawa, Setagaya-ku, Tokyo, 154-8525, Japan}
\email{w3c@komazawa-u.ac.jp}

\subjclass{Primary 57M25; Secondary 57Q35}



\keywords{knots, closed incompressible surface, meridionally small, bridge number, thin position, width, waist, trunk, supertrunk, hull number, 3-width}

\begin{abstract}
We introduce two numerical invariants, the {\em waist} and the {\em trunk} of knots.
The {\em waist} of a closed incompressible surface in the complement of a knot is defined as the minimal intersection number of all compressing disks for the surface in the 3-sphere and the knot.
Then the {\em waist} of a knot is defined as the maximal waist of all closed incompressible surfaces in the complement of the knot.
On the other hand, the {\em trunk} of a knot is defined as the minimal number of the intersection of the most thick level 2-sphere and the knot over all Morse positions of the knot.
In this paper, we obtain an inequality between the waist and the trunk of knots and show that the inequality is best possible.
We also define the {\em supertrunk} of a knot and relate it to the hull number and the 3-width. 
\end{abstract}

\maketitle

\section{Introduction}

\subsection{Waist of knots}
Here, we introduce a new numerical invariant, the {\em waist}, of knots which is obtained from all closed incompressible surfaces in the knot complements.
The waist is a generalization of the wrapping number of satellite knots and therefore it is considered to be a quantity to measure how the knot is tied up to a bunch.
A generalization of the winding number of satellite knots has been done in \cite{MO9}.

Let $K$ be a non-trivial knot in the 3-sphere $S^3$.
Then a torus $\partial N(K)$ is incompressible in $S^3-K$.
Therefore $S^3-K$ contains at least one closed incompressible surface.
Let $F$ be one of closed incompressible surfaces in $S^3-K$.
Since any closed surface in $S^3$ is compressible, there exists a compressing disk $D$ for $F$ in $S^3$.
We may assume that $D$ intersects $K$ transversely.
We call the minimum number of the intersection number $|D\cap K|$ of $D$ and $K$ over all compressing disks $D$ for $F$ in $S^3$ the {\em waist} of $F$, which is denoted by $waist(F)$.
The {\em waist} of $K$ is defined as the maximum number of $waist(F)$ over all closed incompressible surfaces $F$ in $S^3-K$, and we denote it by $waist(K)$.
Namely,
\[
waist(K)=\max_{F} \min_{D} |D\cap K|.
\]
We define that $waist(K)=0$ for the trivial knot $K$.

\begin{example}\label{cable}
Let $C(K;p,q)$ be the $(p,q)$-cable knot of the trefoil knot $K$.
Then, $waist(C(K;p,q))= p$.
In general, $waist(C(K;p,q))\ge p\times waist(K)$ for any knot $K$.
\end{example}

By Example \ref{cable}, the waist is considered to be a quantity to measure how the knot is tied up to a bunch.

\begin{problem}
In general, for a satellite knot $K'$ with a companion knot $K$ and a wrapping number $wrap(K')$ with respect to $K$, does the following equality hold?
\[
waist(K')=wrap(K')waist(K)
\]
\end{problem}

It holds that $waist(K')\ge wrap(K')waist(K)$.
Hence, $waist(K_1 \# K_2)=\max \{waist(K_1),waist(K_2)\}$.

\begin{example}
It follows from \cite{HT} that $waist(K)=1$ for any $2$-bridge knot $K$ since $S^3-K$ contains only peripheral torus as a closed incompressible surface.
On the other hand, there exists a 3-bridge knot $K$ with $waist(K)=2$.
See \cite{MO2} for the characterization of genus two closed incompressible and meridionally incompressible surfaces in the 3-bridge knot complements.

\end{example}

It is known that there are many knot classes whose knot has the waist one;

\begin{itemize}
\item 2-bridge knots (\cite{HT})
\item torus knots (\cite{J})
\item twisted torus knots with twists on 2-strands (\cite{Mor})
\item small knots (include all knot classes above.)
\item alternating knots (\cite{M})
\item almost alternating knots (\cite{A2})
\item toroidally alternating knots (\cite{A})
\item 3-braid knots (\cite{LP})
\item Montesinos knots (\cite{O})
\item algebraically alternating knots (include both of algebraic knots and alternating knots) (\cite{MO3})
\end{itemize}

The waist one knots have the next particular property.
We say that a surface properly embedded in a knot exterior is {\em free} if it cuts the knot exterior into handlebodies.

\begin{theorem}\label{free}
Let $K$ be a knot with $waist(K)=1$.
Then, any incompressible surface properly embedded in the exterior of $K$ with boundary of finite slope is free.
\end{theorem}

By Theorem \ref{free}, the exterior of a waist one knot is composed of one or two handlebodies.
Therefore, the waist one knots have simple complements among all knots.

\begin{problem}
Does the converse of Theorem \ref{free} hold?
Namely, if any incompressible surface properly embedded in the exterior of $K$ with boundary of finite slope is free, then $waist(K)=1$?
\end{problem}


\subsection{Trunk of knots}
Let $h:S^3\to \Bbb{R}$ be a Morse function with two critical points.
Suppose that $h$ is a Morse function on $K$.
We define the {\em trunk} of $K$ as
\[
trunk(K)=\min_{h} \max_{t\in \Bbb{R}} |h^{-1}(t)\cap K|,
\]
where $h$ is over all Morse function with two critical points.
In general, we have $trunk(K)\le 2 bridge(K)$, where $bridge(K)$ denotes the bridge number of $K$.

The next problem seems to be essential on trunk of knots.
Definitions about thin position are given in Section \ref{proof}.

\begin{problem}
If $K$ is in thin position, then does a most thick level 2-sphere give $trunk(K)$?
\end{problem}

We answer this problem partially by Theorem \ref{m-small}.

A knot $K$ is called {\em meridionally small} if there exists no essential surface in the exterior $E(K)$ of $K$ with meridional boundary.
We remark that by \cite{CGLS}, small knots are meridionally small and the converse holds if $w(K)=1$.

\begin{theorem}\label{m-small}
If a knot $K$ is meridionally small, then $trunk(K)=2bridge(K)$.
\end{theorem}

It holds that $trunk(K_1\# K_2)\le \max \{ trunk(K_1), trunk(K_2)\}$

\begin{problem}\label{sum}
Does the following equality hold?
\[
trunk(K_1\# K_2)= \max \{ trunk(K_1), trunk(K_2)\}
\]
\end{problem}

We answer Problem \ref{sum} partially.

\begin{theorem}\label{m-small2}
If $K_1$ and $K_2$ are meridionally small knots, then
\[
trunk(K_1\# K_2)= \max \{ trunk(K_1), trunk(K_2)\}
\]
\end{theorem}

\subsection{Relation between waist and trunk}
In this paper, we obtain an inequality between the waist and the trunk of knots.
\begin{theorem}\label{main}
For a knot $K$,
\[
\displaystyle waist(K)\le \frac{trunk(K)}{3}
\]
\end{theorem}

Theorem \ref{main} says that there exists a relation between a {\em nature of surfaces} in the knot complements (waist) and a {\em nature of positions} of knots (trunk).

\begin{remark}\label{best}
Theorem \ref{main} is best possible.
In fact, there exists a knot $K$ which satisfies $waist(K)=\lfloor 2bridge(K)/3\rfloor$ for any integer $bridge(K)$, where $\lfloor x\rfloor$ denotes the floor function.

Let $K$ be a 3-bridge knot and $F$ a closed incompressible surface of genus two in the left side of Figure \ref{ex1}.
Then, it holds that $bridge(K)=3$ and $waist(F)=2$.
Moreover, let $K^n$ be a $(n,*)$-cable knot of $K$, and then we have that $bridge(K^n)=3n$ and $waist(F)=2n$.
Hence, Theorem \ref{main} is best possible when $bridge(K)\equiv 0\ (\rm{mod}\ 3)$.
Also, since it holds in general that $waist(K_1\# K_2)=\max\{waist(K_1),waist(K_2)\}$, it holds that for any 2-bridge knot $K'$, $bridge(K^n\# K')=3n+1$ and $waist(F)=2n$.
Hence, Theorem \ref{main} is also best possible when $bridge(K)\equiv 1\ (\rm{mod}\ 3)$.

To show that Theorem \ref{main} is best possible in the case $bridge(K)\equiv 2\ (\rm{mod}\ 3)$, let $K$ be a 3-bridge knot and $F$ a closed incompressible surface of genus five in the right side of Figure \ref{ex1}.
Then, it holds that $bridge(K)=3$ and $waist(F)=2$.
Moreover, let $K^n$ be a $(n,*)$-cable knot of $K$, and then we have that $bridge(K^n)=3n$ and $waist(F)=2n$.
We construct a knot ${K^n}'$ from $K^n$ by making two bridges parallelism from the left among three bridges of $K$ and doing adequate twists so that ${K^n}'$ becomes the knot.
Then, it holds that $bridge(K)=3n+2$ and $waist(F)=2n+1$.
Hence, Theorem \ref{main} is also best possible when $bridge(K)\equiv 2\ (\rm{mod}\ 3)$.

\begin{figure}[htbp]
	\begin{center}
	\begin{tabular}{cc}
	\includegraphics[trim=0mm 0mm 0mm 0mm, width=.2\linewidth]{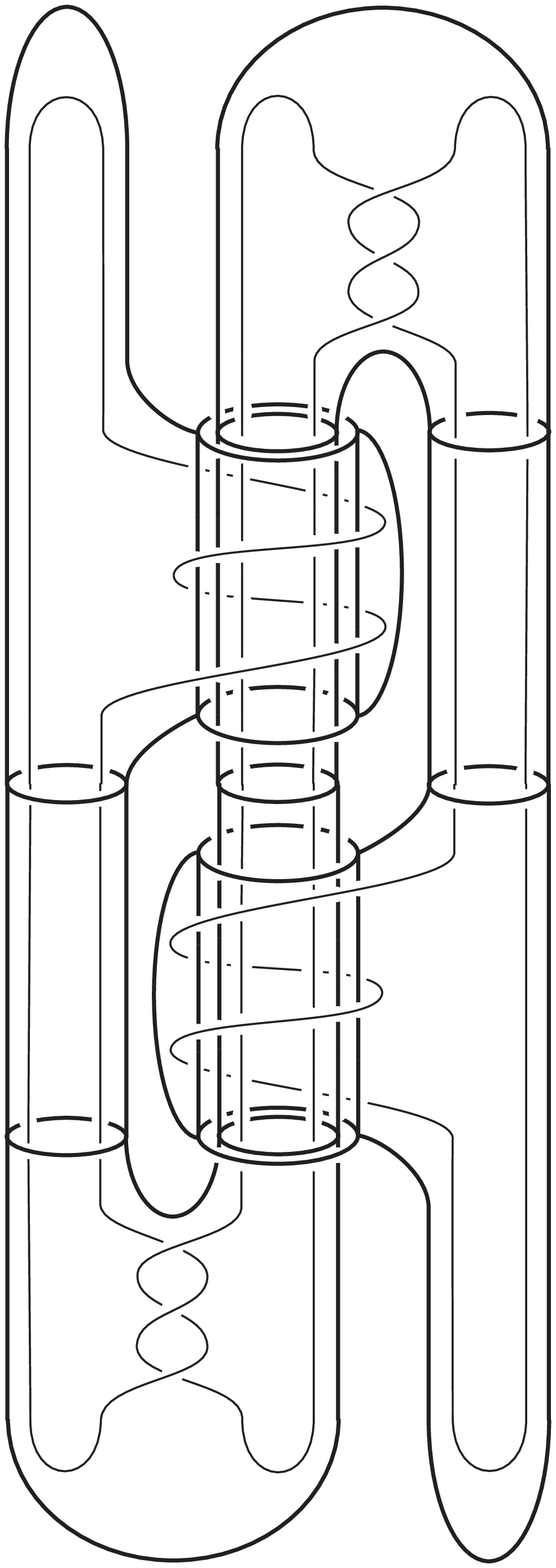}&
	\includegraphics[trim=0mm 0mm 0mm 0mm, width=.2\linewidth]{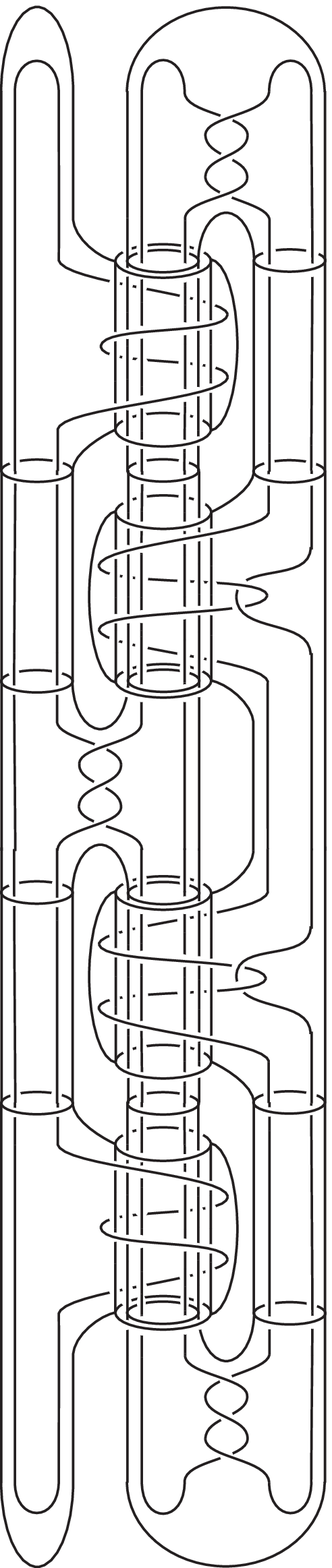}\\
	\end{tabular}
	\end{center}
	\caption{3-bridge knots and closed incompressible surfaces}
	\label{ex1}
\end{figure}
\end{remark}


\begin{remark}
An inequality in Theorem \ref{main} can have an arbitrarily gap against Remark \ref{best}.
For example, let $K$ be a $(p,q)$-torus knot.
Since $K$ is (meridionally) small, $waist(K)=1$ and by \cite{S} or \cite{S3} and Theorem \ref{m-small}, $trunk(K)=2bridge(K)=2\min\{p,q\}$.
\end{remark}

\subsection{Conjectures}
A knot diagram is {\em $m$-almost alternating} if it becomes to be alternating after changing some $m$-crossings, and a knot is said to be {\em $m$-almost alternating} if it has an $m$-almost alternating diagram but no $m-1$-almost alternating diagram.
The next inequality is expected from \cite{A2}.

\begin{conjecture}
Let $K$ be an $m$-almost alternating knot $(m\ge 1)$.
Then, $waist(K)\le m$.
\end{conjecture}

Also the next inequality is expected from \cite{LP}.

\begin{conjecture}
For any knot $K$,
\[
\displaystyle waist(K)\le \frac{braid(K)}{2},
\]
where $braid(K)$ denotes the braid index of $K$.
\end{conjecture}

\subsection{Supertrunk and hull number, 3-width}
We can define the {\em supertrunk} of a knot $K$ in $\Bbb{R}^3$ as
\[
supertrunk(K)=\min_{K} \max_{P\in \mathcal{P}} |P\cap K|,
\]
where $\mathcal{P}$ denotes the set of all plane in $\Bbb{R}^3$ and the minimum is taken over all knots equivalent to $K$.
In general, we have $trunk(K)\le supertrunk(K)$ and $supertrunk(K)\le 2 superbridge(K)$, where $superbridge(K)$ denotes the superbridge index of $K$ (\cite{K}).

On the other hand, the {\em $n$-th hull} $h_n(K)$ of a fixed knot $K$ is the set of points $p\in\Bbb{R}^3$ such that every plane through $p$ cuts $K$ at least $2n$ times (\cite{CKKS}).
Then the {\em hull number} of a knot $K$ is defined as
\[
hull(K)=\min_{K} \max\{n|h_n(K)\ne \emptyset \},
\]
where the minimum is taken over all knots equivalent to $K$.
The next inequality is essentially contained in \cite[Proposition in Section 5]{CKKS}.

\begin{proposition}
For any knot $K$ in $\Bbb{R}^3$,
\[
\displaystyle hull(K)\le \frac{supertrunk(K)}{2},
\]
where $hull(K)$ denotes the hull number of $K$.
\end{proposition}

The trunk and supertrunk also give rough estimations on the width and 3-width.
See Section \ref{proof} for the definition of the width and \cite{HRT} for the 3-width.

\begin{proposition}
For a knot $K$ in $\Bbb{R}^3$,
\[
\displaystyle width(K)\ge \frac{trunk(K)^2}{2},
\]
\[
\text{3-}width(K)\ge supertrunk(K)(supertrunk(K)+1).
\]
\end{proposition}

\section{Proofs}\label{proof}

\begin{proof} (of Theorem \ref{free})
Suppose that there exists an incompressible and $\partial$-incompressible surface $F$ of finite slope in the exterior $E(K)$ of a knot $K$ with $waist(K)=1$.
Then at least one of components obtained by cutting $E(K)$ by $F$ is not a handlebody and hence by \cite[Lemma 2.2]{MO9} there exists a closed incompressible surface $S$ in $E(K)-F$.
It follows from $waist(K)=1$ that there exists an annulus $A$ connecting $S$ and a meridian of $E(K)$.
By an argument similar to the proof of Theorem 1 in \cite{MO8} and the incompressibility and $\partial$-incompressibility of $F$, we may assume that $A\cap F=\emptyset$.
However, this means that the $\partial$-slope of $F$ is infinite (meridional).
\end{proof}

Let $h:S^3\to \Bbb{R}$ be the standard Morse function and $K$ be a knot in $S^3$ which is assumed so that the restriction of $h$ to $K$ is a Morse function.
Let $t_1,\ldots,t_n$ be critical values of $h|_K$ such that $t_i<t_{i+1}$, and choose regular values $r_1,\ldots,r_{n+1}$ of $h|_K$ so that $r_i<t_i<r_{i+1}$.
The {\em width} $w(K)$ of $K$ is defined as
\[
w(K)=\min_{K} \sum_{i=1}^{n+1}|h^{-1}(r_i)\cap K|,
\]
where the minimum is taken over all knots equivalent to $K$ (\cite{G}).
A knot $K$ is in {\em thin position} if it realizes $w(K)$.
We say that a level 2-sphere $S=h^{-1}(r_i)$ is {\em thick} if $|h^{-1}(r_{i-1})\cap K|<|h^{-1}(r_{i})\cap K|$ and $|h^{-1}(r_{i+1})\cap K|<|h^{-1}(r_{i})\cap K|$, and that a level 2-sphere $S=h^{-1}(r_i)$ is {\em thin} if $|h^{-1}(r_{i-1})\cap K|>|h^{-1}(r_{i})\cap K|$ and $|h^{-1}(r_{i+1})\cap K|>|h^{-1}(r_{i})\cap K|$.
Thompson (\cite{T}) showed that if $K$ is in thin position, then it is in a bridge position or not meridionally small, and Wu (\cite{W}) extended this result by showing that if $K$ is in thin position, then it is in a bridge position or a most thinnest thin level 2-sphere is incompressible in the complement of $K$.

We have the following lemma by an argument same as the proof of \cite[Theorem 1]{W}.

\begin{lemma}\label{isotopy}
Let $K$ be a knot in $S^3$ such that $h|_K$ is a Morse function and $P$ a thinnest level 2-sphere.
If $P-K$ is compressible in $S^3-K$, then there exists a (``horizontal'' and ``vertical'') isotopy $f_t$ of $S^3$ such that $w(K')<w(K)$ and $trunk(K')\le trunk(K)$, where $K'=f_1(K)$.
\end{lemma}

\begin{proof} (of Theorem \ref{m-small})
Let $K$ be a knot in $S^3$ such that $h|_K$ is a Morse function and $|h^{-1}(t)\cap K|\le trunk(K)$ for all $t\in\Bbb{R}$.
If there exists no thin level 2-sphere for $K$, then $K$ is in a bridge position and we have $trunk(K)\ge 2bridge(K)$.
Otherwise, let $P$ be a thinnest level 2-sphere for $K$.
If $|P\cap K|=2$, then $P$ decomposes $K$ into two knots $K_1$ and $K_2$.
However since $K$ is meridionally small, at least one of $K_1$ and $K_2$ is trivial.
Therefore, hereafter we assume that $|P\cap K|\ge 4$.
If $P-K$ is incompressible in $S^3-K$, then this contradicts that $K$ is meridionally small.
Otherwise, by Lemma \ref{isotopy}, we can isotope $K$ so that the width of $K$ decreases without increasing the trunk of $K$.
In this case, we can proceed by an induction on $w(K)$.
\end{proof}

\begin{proof} (of Theorem \ref{m-small2})
We have by Theorem \ref{m-small} that $trunk(K_i)=2bridge(K_i)$ $(i=1,2)$ and we may assume that $trunk(K_i)\ge 4$ for $i=1,2$.
Put $K_1\# K_2$ so that $h|_{K_1\# K_2}$ is a Morse function and $|h^{-1}(t)\cap (K_1\# K_2)|\le trunk(K_1\# K_2)$ for all $t\in\Bbb{R}$.
Moreover we assume that the width of $K_1\# K_2$ is minimal under this supposition.

If there exists no thin level 2-sphere for $K_1\# K_2$, then $K_1\# K_2$ is in a bridge position and we have $trunk(K_1\# K_2)= 2bridge(K_1\# K_2)$.
By the additivity of the bridge number (\cite{S}) and Theorem \ref{m-small}, $2bridge(K_1\# K_2)=2bridge(K_1)+2bridge(K_2)-2=trunk(K_1)+trunk(K_2)-2$.
Hence $trunk(K_1\# K_2)> \max\{trunk(K_1),trunk(K_2)\}$, a contradiction.

Otherwise, let $P$ be a thinnest level 2-sphere for $K_1\# K_2$.
If $|P\cap (K_1\# K_2)|=2$, then $P$ is a decomposing sphere for $K_1\# K_2$.
In this case, we have $trunk(K_1\# K_2)=\max \{2bridge(K_1),2bridge(K_2)\}=\max\{trunk(K_1), trunk(K_2)\}$.
If $|P\cap (K_1\# K_2)|\ge 4$, then the miniality of the width of $K_1\# K_2$ and Lemma \ref{isotopy}, $P-(K_1\# K_2)$ is incompressible in $S^3-(K_1\# K_2)$.
However this implies that there exists an essential surface with meridional boundary in one of the exterior $E(K_i)$ for $i=1,2$, and contradicts that both of $K_1$ and $K_2$ are meridionally small.
\end{proof}




Let $h:S^3\to \Bbb{R}$ be the standard Morse function and $F$ be a closed surface in $S^3$ which is assumed so that the restriction of $h$ to $F$ is a Morse function.
Suppose that $F$ has a saddle point $p$ which corresponds the critical value $t_p\in \Bbb{R}$.
Let $X_p$ be a pair of pants component of $F\cap h^{-1}([t_p-\epsilon, t_p+\epsilon])$ containing $p$ for a fixed sufficiently small positive real number $\epsilon$.
Let $C_1$, $C_2$ and $C_3$ be the boundary components of $X_p$, where we assume that $C_1$ and $C_2$ are contained in the same level $h^{-1}(t_p\pm\epsilon)$, and $C_3$ is contained in the another level $h^{-1}(t_p\mp\epsilon)$.
See Figure \ref{saddle}.
We call a saddle point $p$
\begin{enumerate}
\item {\em Type I} if all of $C_1$, $C_2$ and $C_3$ are inessential in $F$,
\item {\em Type II} if exactly one of $C_1$ and $C_2$ is essential and $C_3$ is essential in $F$,
\item {\em Type III} if both of $C_1$ and $C_2$ are essential and $C_3$ is inessential in $F$,
\item {\em Type IV} if all of $C_1$, $C_2$ and $C_3$ are essential in $F$,
\end{enumerate}

\begin{figure}[htbp]
	\begin{center}
	\includegraphics[trim=0mm 0mm 0mm 0mm, width=.5\linewidth]{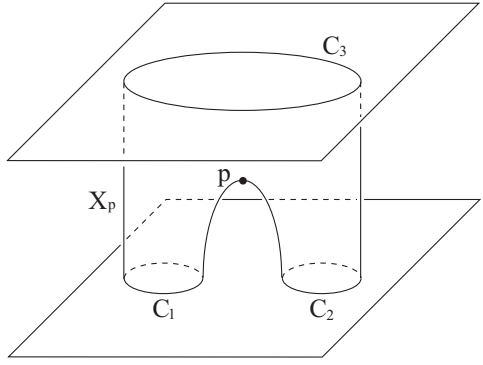}
	\end{center}
	\caption{a pair of pants component $X_p$ and three boundary components $C_1$, $C_2$ and $C_3$}
	\label{saddle}
\end{figure}

\begin{lemma}\label{eliminate}
Let $F$ be a closed surface in $S^3$ such that $h|_F$ is a Morse function.
Then we can remove all saddle points of Type I and II without alterations of the number of saddle points of Type III and IV.
\end{lemma}

\begin{proof}
By an induction on the number of saddle points of Type I and II, we remove all saddle points of Type I and II.
Suppose without loss of generality that $C_1$ is inessential in $F$ and bounds a disk $D$ in $F$.
Furthermore we may assume that $D$ contains no saddle point of $F$ and hence only one maximal or minimal point.
As in the Proof of Lemma 5 in \cite{MO2}, we can eliminate the maximal/minimal point of $D$ and the saddle point $p$ by a ``vertical'' isotopy of $D$.
We remark that there is no birth nor death of other critical points for this elimination of two critical points.
Repeating this process, all saddle points of Type I and II can be removed.
\end{proof}

\begin{lemma}\label{saddle}
Let $F$ be a closed surface in $S^3$ such that $h|_F$ is a Morse function.
Then 
\begin{enumerate}
\item $F$ is a 2-sphere if and only if any saddle point of $F$ is of Type I.
\item $F$ is a torus if and only if any saddle point of $F$ is Type I, II or III and there exists a saddle point of Type III.
\item $F$ is a closed surface of genus greater than one if and only if any saddle point of $F$ is Type I, II, III or IV and there exists a saddle point of Type IV.
\end{enumerate}
\end{lemma}

\begin{proof}
First by Lemma \ref{eliminate}, we remove all saddle points of Type I and II.

(1) If $F$ is a 2-sphere, then any loop contained in $F$ is inessential.
Therefore any saddle point of $F$ is of Type I.
Conversely if any saddle point is of Type I, then there exists no saddle point after we eliminated all saddle points of Type I.
Thus $F$ is a 2-sphere.

(2) If $F$ is a torus, then any essential loops in $F$ are mutually parallel and non-separating in $F$.
If two of three loops appearing in a saddle point of Type IV are mutually parallel in $F$, then the other one is separating in $F$.
Hence $F$ contains no saddle point of Type IV when $F$ is a torus.
Moreover after we eliminated all saddle points of Type I and II, any saddle point of $F$ is of Type III and there exists a saddle points of Type III.

(3) If $F$ is a closed surface of genus greater than one, then any saddle point of $F$ is of Type III or IV after we eliminated all saddle points of Type I and II.
If any saddle point of $F$ is of Type III, then $F$ can be decomposed into annuli and it is a torus, a contradiction.
Hence there exists a saddle point of Type IV.
\end{proof}

\begin{proof} (of Theorem \ref{main})
In a case that $K$ is a trivial knot, $waist(K)=0$ and $trunk(K)=2$, so we have an inequality of Theorem \ref{main}.

Let $K$ be a non-trivial knot, $F$ be a closed incompressible surface in $S^3-K$ and $h:S^3\to \Bbb{R}$ be a Morse function with two critical points which is also a Morse function on $K$ and $F$.
Suppose that $h$ satisfies $trunk(K)=max_{t\in\Bbb{R}}|h^{-1}(t)\cap K|$.
To prove Theorem \ref{main}, we need to consider two cases.

\begin{description}
\item [Case I] There exists a saddle point of Type IV in $F$.
\item [Case II] There exists no saddle point of Type IV in $F$.
\end{description}

Case I: If there exists a saddle point $p$ of Type IV, then each essential loop $C_i$ bounds a disk $D_i$ in $h^{-1}(t_p\pm\epsilon)$ so that $D_1\cup D_2\cup D_3\cup X_p$ forms a 2-sphere which is isotopic to a level 2-sphere $h^{-1}(t_p)$, where $t_p$ is the critical value corresponding to $p$ and $X_p$ is a pair of pants component of $F\cap h^{-1}([t_p-\epsilon, t_p+\epsilon])$ containing $p$.
See Figure \ref{2-sphere}.
We remark that there exists no critical point of $K$ in $h^{-1}([t_p-\epsilon, t_p+\epsilon])$ since we chose a fixed sufficiently small positive real number $\epsilon$.
Then we have
\[
|D_1\cap K|+|D_2\cap K|+|D_3\cap K|=|h^{-1}(t_p)\cap K|.
\]
Since $|h^{-1}(t_p)\cap K|\le trunk(K)$ for all $t_p\in \Bbb{R}$,
\[
\min_i|D_i\cap K|\le \frac{trunk(K)}{3}.
\]

\begin{figure}[htbp]
	\begin{center}
	\includegraphics[trim=0mm 0mm 0mm 0mm, width=.5\linewidth]{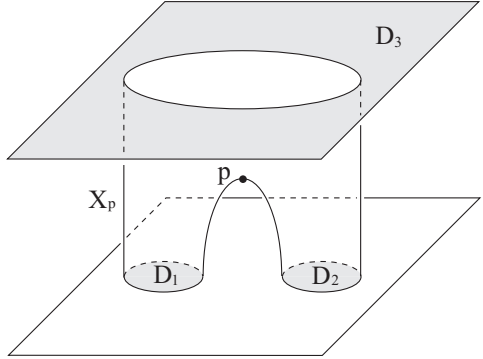}
	\end{center}
	\caption{$D_1\cup D_2\cup D_3\cup X_p$ forms a 2-sphere}
	\label{2-sphere}
\end{figure}

In the following, we show that $waist(K)\le |D_i\cap K|$.
Let $C_i'$ be a loop of $D_i\cap F$ which is essential in $F$ and innermost in $D_i$, and $D_i'$ be the innermost disk in $D_i$ which is bounded by $C_i'$.
Since $D_i\cap F$ contains at least one loop which is essential in $F$ (e.g. $\partial D_i=C_i$), we can take such a loop $C_i'$ and a disk $D_i'$.
Then by an elementary cut and paste argument, we may assume that $D_i'\cap F=\partial D_i'=C_i'$.
Since $D_i'$ is also a compressing disk for $F$ in $S^3$, we have
\[
waist(K)\le |D_i'\cap K|\le |D_i\cap K|.
\]

Case II: We note that $F$ is not a 2-sphere since $K$ is a knot.
If there exists no saddle point of Type IV, then by Lemma \ref{saddle}, $F$ is a torus.
If there exists a level 2-sphere $S=h^{-1}(t)$ for some $t\in\Bbb{R}$ such that $S$ contains at least three non-parallel classes of loops of $F\cap S$ which are essential in $F$, then we obtain an inequality of Theorem \ref{main} by the previous argument in Case I.
So we assume that
\begin{description}
\item [Condition A] for each level 2-sphere $S$, every essential loops in $F$ are mutually parallel in $S$.
\end{description}
We will show that if Condition A holds, then
\begin{description}
\item [Condition B] $F$ can be isotoped so that it is a ``1-bridge torus'',
\end{description}
where ``1-bridge'' means that it has only one maximal and one minimal point, and two saddle point of Type III.
Note that we do not need to consider $K$ in the next claim.

\begin{claim}\label{1-bridge}
If Condition A satisfies, then Condition B holds.
\end{claim}

Before we prove Claim \ref{1-bridge}, we need some definitions.
Let $F$ be a closed surface in $S^3$ such that $h|_F$ is a Morse function.
Let $t_1,\ldots,t_n$ be critical values of $F$ such that $t_i<t_{i+1}$, and choose regular values $r_1,\ldots,r_{n+1}$ of $F$ so that $r_i<t_i<r_{i+1}$.
We say that a level 2-sphere $S=h^{-1}(r_i)$ is {\em thick} if $|h^{-1}(r_{i-1})\cap F|<|h^{-1}(r_{i})\cap F|$ and $|h^{-1}(r_{i+1})\cap F|<|h^{-1}(r_{i})\cap F|$, and that a level 2-sphere $S=h^{-1}(r_i)$ is {\em thin} if $|h^{-1}(r_{i-1})\cap F|>|h^{-1}(r_{i})\cap F|$ and $|h^{-1}(r_{i+1})\cap F|>|h^{-1}(r_{i})\cap F|$.
Then the pair $(S^3,F)$ can be decomposed by thick level 2-spheres and thin level 2-spheres, and put $(S^3,F)=(M_1,F_1)\cup_{S_1}(M_2,F_2)\cup_{S_2}\cdots\cup_{S_{m-1}}(M_m,F_m)$, where $S_i$ is a thick (resp. thin) level 2-sphere if $i$ is odd (resp. even).

For example, the genus two closed incompressible surface $F$ in the left side of Figure \ref{ex1} has two thick level 2-spheres and one thin level 2-spheres, and $(S^3,F)$ is decomposed into four regions.
On the other hand, the genus five closed incompressible surface $F$ in the right side of Figure \ref{ex1} has four thick level 2-spheres and three thin level 2-spheres, and $(S^3,F)$ is decomposed into eight regions.

\begin{proof} (of Claim \ref{1-bridge})
First by Lemma \ref{eliminate}, we remove all saddle points of Type I and II.
Then any critical point of $F$ is of Type III or maximal or minimal.
Thus an inessential loop $C_3$ appearing in near a saddle point of Type III bounds a disk in $F$ which contains only one maximal or minimal point.
Hence each component of the thin/thick decomposition of $F$ is an annulus containing one saddle point of Type III and one maximal or minimal point, or an annulus containing no critical point.
Therefore, in each region $M_i$ of the thin/thick decomposition of $S^3$, each component of $F_i$ is a boundary parallel annulus or a ``vertical'' annulus.
Here we remark that at least one component of $F_i$ is a boundary parallel annulus.

We turn our attention to the highest thick level 2-sphere $S_{m-1}$ and two regions $M_{m-1}$ and $M_m$.
We remark that all components of $F_m$ are boundary parallel to $S_{m-1}$ in $M_m$, all annulus components of $F_{m-1}$ containing no critical point are vertical annuli in $M_{m-1}$, and all annulus components of $F_{m-1}$ are boundary parallel to $S_{m-1}$ in $M_{m-1}$ by Condition A.
Let $A$ be an outermost annulus in $M_{m-1}$ which is boundary parallel to $S_{m-1}$.
We push in $A$ into $M_m$, then we have two cases.

\begin{description}
\item[Case 1] $A$ connects two annulus components $A_1$ and $A_2$ of $F_m$.
\item[Case 2] $A$ connects one annulus component $A'$ of $F_m$.
\end{description}

Case 1: By an isotopy, we may assume that an annulus $A\cup A_1\cup A_2$ contains one saddle point of Type III and one maximal point.
By an induction on the number of critical points of $F$, we will arrive in Case 2.

Case 2: $A\cup A'$ forms the whole surface $F$ which is in a 1-bridge position, and we have Condition B.
\end{proof}

However, Claim \ref{1-bridge} shows that $F$ is compressible in $S^3-K$ since $F$ cannot be in both sides of $F$.
\end{proof}

\begin{remark}
By the proof of Theorem \ref{main}, for any knot $K$ and any closed incompressible surface $F$ in $S^3-K$ which are in Morse position with respect to the height function $h:S^3\to \Bbb{R}$, there exists a level 2-sphere $P=h^{-1}(t)$ for some $t\in\Bbb{R}$ such that $F\cap P$ contains at least three loops which are mutually non-parallel on $P$ and essential in $F$.
\end{remark}

\section{Further way}
In this paper, we defined the waist and the trunk of knots and obtained an inequality between them.
Both numerical invariants are considered as ``horizontal'' quantity.

To define a ``vertical'' quantity, we decompose the triple $(S^3,F,K)$ by thin and thick level 2-spheres for $F$ as in Section 2.
Then it seems to be reasonable that we define the {\em height} of $F$, $height(F)$, as the maximal number of thick level 2-spheres (or the maximal number of thin level 2-spheres $+1$) under the condition that the number of maximal and minimal points is minimized.
The given condition in this definition implies that there exists no saddle point of Type I nor II.
Then we can define the {\em height} of a non-trivial knot $K$ as
\[
height(K)=\max_{\mathcal{F}} \{height(\mathcal{F})\},
\]
where $\mathcal{F}$ is a family of mutually disjoint closed incompressible surfaces in $S^3-K$ whose number is maximal.
The author believe that the height of a knot (as same as the ``maximal'' thin/thick decomposition) will give more close nature than thin position.

\bigskip


\bibliographystyle{amsplain}

\end{document}